\documentclass[12pt,leqno,twoside]{article}

\usepackage{amsmath,amssymb,amsthm,rotating,}
\usepackage[left=.8in,right=.5in,top=.5in,]{geometry}
\usepackage{xcolor,color,fancyhdr,}

\usepackage{enumerate}

\usepackage{hyperref,xifthen,iftex,}
\usepackage{biblatex}
\hypersetup{linktocpage,unicode,pdfencoding=auto,pdfstartview=,allcolors=black,anchorcolor=black,allbordercolors=black,hidelinks}
\DeclareNameAlias{sortname}{given-family}

\newtheorem{theorem}{\textit{Theorem}}
\newtheorem{lemma}{\textit{Lemma}}

\theoremstyle{remark}
\newtheorem*{remark}{\normalfont\textit{Remark}}

\usepackage{scalerel}
\newcommand{\floor}[1]{\ifxetex{\squarebracket{#1}}\else{\left\lfloor{#1}\right\rfloor}\fi}

\ifxetex\fi
\newcommand{\parenthesis}[1]{\ifxetex{\scaleleftright[.8ex]{(}{#1}{)}}\else{\left(#1\right)}\fi}

\newcommand{\squarebracket}[1]{\ifxetex{\scaleleftright[.8ex]{[}{#1}{]}}\else{\left[#1\right]}\fi}
\newcommand{\bigo}[1]{O\parenthesis{#1}}
\newcommand{\func}[2]{#1\parenthesis{#2}}

\usepackage{extramarks}

\AtBeginEnvironment{theorem}{\addtocounter{theorem}{1}%
	\extramarks{\thetheorem}{}\addtocounter{theorem}{-1}}
\fancypagestyle{fancy}{
	\fancyhf{}
	\fancyhead[LO]{\ifthenelse{0\firstleftxmark=0}
		{}% no theorems yet, so no header
		{\ifthenelse{\equal{\firstleftxmark}{\lastleftxmark}}
			{\textit{Theorem} \firstleftxmark}% Only one theorem on the page
			{\textit{Theorems} \firstleftxmark–\lastleftxmark}}}
	\fancyhead[C]{\small\itshape Masum Billal}
	\fancyhead[R]{\thepage}
	\fancyhead[LE]{\small\itshape \thesection.\hspace{.2em}\nouppercase\rightmark}
}

\author{Masum Billal}
\title{Asymptotic Result of A Generalization of A GCD-Sum}

\begin{document}
	\pagestyle{empty}
	\maketitle
		\begin{abstract}
			In a paper in the American Mathematical Monthly, the corresponding author asks for an asymptotic of a gcd-sum function
				\begin{align}
					\sum_{ab\leq N}\tau(\gcd(a,b))\label{eqn:taugcdsum}
				\end{align}
			We extensively study generalizations of the sum in the aforementioned paper and establish asymptotic results using elementary methods only.
		\end{abstract}
	\textcite[pp. 3, Remark 1]{dudek_2016} establishes that the probability of a random integer (not necessarily prime) satisfying Euclid's lemma is almost zero; however, in doing so, the corresponding author remarks about \eqref{eqn:taugcdsum} that \textit{it would be interesting to see an asymptotic formula for this}. Theorems \ref{thm:generalizationA}--\ref{thm:asympD} will establish the asymptotic for the first generalization. As we will show, this is intimately related to the generalization of the so-called \textit{Pillai arithmetic function} studied in \textcite{pillai_1933}. There has been extensive studies on gcd-sum functions afterwards. \textcite{toth_2010} gives a survey on gcd-sum functions. \textcite{broughan_2001}, \textcite{haukkanen_2008}, \textcite{luca_thangadurai_2009}, \textcite{toth_2013} cover some discussions and generalizations on Pillai functions. However, the results we consider in this paper seem to have not been established in the literature.
	\section{Generalized Convolution}
	Consider an arithmetic function $f:\mathbb{N}^{k}\to\mathbb{R}$ in $k$ variables $x_{1},\ldots,x_{k}\in\mathbb{R}$. For brevity, let us write $f(\mathbf{x})=f(x_{1},\ldots,x_{k})$, $n^{\mathbf{s}}\leq \mathbf{x}$ denote $n^{s_{i}}\leq x_{i}$ for $s_{1},\ldots,s_{k}\leq\mathbb{R},n\in\mathbb{N}$ and $\func{f}{\frac{\mathbf{x}}{\mathbf{s}}}$ denote $\func{f}{\frac{x_{1}}{n^{s_1}},\ldots,\frac{x_{k}}{s_{k}}}$. Then we generalize the concept of generalized convolution (see \textcite[II, $\S$2.14]{apostol_1976}) as follows
		\begin{align*}
			(\alpha\circ g)(\mathbf{x},\mathbf{s})
				& = \sum_{n^{s_{i}}\leq x_{i}}\alpha(n)\func{f}{\frac{\mathbf{x}}{n^{\mathbf{s}}}}
		\end{align*}
	where $\alpha$ is an  arithmetic function. It can be easily verified that the law of associativity and inversion applies for $\circ$ as it does for generalized convolution $\circ$. That is,
		\begin{align}
			\alpha\circ(\beta\circ f)
				& = (\alpha\ast\beta)\circ f\label{eqn:associativity}
		\end{align}
	where $\alpha\ast\beta$ is the Dirichlet convolution of arithmetic functions $\alpha$ and $\beta$. Also, if $g=\alpha\circ f$ such that $\alpha(1)\neq0$ and $\alpha^{-1}$ is the Dirichlet inverse of $\alpha$ for fixed $\mathbf{x}$ and $\mathbf{s}$, then
		\begin{align}
			g
				& = \alpha^{-1}\circ f\label{eqn:geninv}
		\end{align}
	Next, we consider a variation of $\circ$.
		\begin{align*}
			(f\bullet g)(x)
				& = \sum_{n\leq x}f(n)\func{g}{\floor{\frac{x}{n}}}
		\end{align*}
	Using the fact that there are at most $2\floor{\sqrt{n}}$ distinct numbers in $\floor{\frac{x}{1}},\ldots,\floor{\frac{x}{n}}$, it can be easily shown that
		\begin{theorem}\label{thm:bullet}
			Let $f$ and $g$ be arithmetic functions such that $F(x)=\sum_{n\leq x}f(n)$. Then
				\begin{align}
					(f\bullet g)(x)
						& = \sum_{n\leq\sqrt{x}}\func{g(n)}{\func{F}{\floor{\frac{x}{n}}}-\func{F}{\floor{\frac{x}{n+1}}}}+\sum_{n\leq x/(\sqrt{x}+1)}f(n)\func{g}{\floor{\frac{x}{n}}}\label{eqn:bullet}
				\end{align}
		\end{theorem}
	Alternatively, we can write
		\begin{align*}
			(f\bullet g)(x)
				& = \sum_{n\leq\sqrt{x}}f(n)\func{g}{\floor{\frac{x}{n}}}+\func{g(n)}{\func{F}{\floor{\frac{x}{n}}}-\func{F}{\floor{\frac{x}{n+1}}}}-\epsilon(x)
		\end{align*}
	where
		\begin{align*}
			\epsilon(x)
				& = 
				\begin{cases}
					f(\floor{\sqrt{x}})\func{g}{\floor{\frac{x}{\floor{\sqrt{x}}}}}& \mbox{ if }x<\floor{\sqrt{x}}(\floor{\sqrt{x}}+1)\\
					0& \mbox{ otherwise}
				\end{cases}
		\end{align*}
	This result is similar to Dirichlet hyperbola method. Some variations of this result have been used in the literature. For example, \textcite{deleglise_rivat_1996} uses a variation of this result to calculate $\sum_{n\leq x}\mu(n)$ efficiently. \textcite{lehman_1960} also relies on a similar result. Unfortunately, there is no proper exposition on this despite being useful in so many cases. Note that \eqref{eqn:associativity} and \eqref{eqn:geninv} applies for $\bullet$ as well.
		\begin{lemma}\label{lem:bulletasymp}
			Let $f$  be an arithmetic function such that
				\begin{align*}
					\sum_{n\leq x}f(n)
						& = O(x^{\xi})
				\end{align*}
			If $g(x)=x^{k}$, then
				\begin{align*}
					(f\bullet g)(x)
						& =
						\begin{cases}
							\bigo{x^{k+1}\log{x}}& \mbox{ if }k+1=\xi\\
							\bigo{x^{\frac{k+\xi+1}{2}}+x^{\xi}\zeta(\xi-k)}&\mbox{ if }k<\xi,k+1\neq\xi\\
							\bigo{x^{k+\frac{1}{2}}}&\mbox{ if }k=\xi\\
							\bigo{x^{\frac{k+\xi+1}{2}}}&\mbox{ if }k>\xi
						\end{cases}
				\end{align*}
		\end{lemma}
	
		\begin{proof}
			We will use the following well known results.
				\begin{align*}
					\sum_{n\leq x}\frac{1}{n}
						& = \log{x}+C+\func{O}{\frac{1}{x}}
				\end{align*}
			If $s\neq 1$ and $s>0$,
				\begin{align*}
					\sum_{n\leq x}\frac{1}{n^{s}}
						& = \frac{x^{1-s}}{1-s}+\zeta(s)+O(x^{-s})
				\end{align*}
			If $s\geq0$,
				\begin{align*}
					\sum_{n\leq x}n^{s}
						& = \frac{x^{s+1}}{s+1}+O(x^{s})
				\end{align*}
			Then
				\begin{align*}
					(f\bullet g)(x)
						& = \sum_{n\leq x}f(n)\floor{\dfrac{x}{n}}^{k}\\
						& = \sum_{n\leq \sqrt{x}}n^{k}\parenthesis{F\parenthesis{\floor{\dfrac{x}{n}}}-F\parenthesis{\floor{\dfrac{x}{n+1}}}}+f(n)\floor{\dfrac{x}{n}}^{k}-\epsilon(x)\\
						& = \sum_{n\leq \sqrt{x}}n^{k}\bigo{\floor{\dfrac{x}{n}}^{\xi}}+f(n)\floor{\dfrac{x}{n}}^{k}-\epsilon(x)\\
						& = x^{\xi}\bigo{\sum_{n\leq \sqrt{x}}\dfrac{n^{k}}{n^{\xi}}}+\sum_{n\leq \sqrt{x}}f(n)\parenthesis{\parenthesis{\dfrac{x}{n}}+O(1)}^{k}-\epsilon(x)\\
						& = x^{\xi}\bigo{\sum_{n\leq \sqrt{x}}\dfrac{n^{k}}{n^{\xi}}}+\sum_{n\leq \sqrt{x}}f(n)\parenthesis{\parenthesis{\dfrac{x}{n}}^{k}+\bigo{\parenthesis{\dfrac{x}{n}}^{k-1}}}-\epsilon(x)\\
						& = x^{\xi}\bigo{\sum_{n\leq \sqrt{x}}\dfrac{n^{k}}{n^{\xi}}}+x^{k}\sum_{n\leq \sqrt{x}}\dfrac{f(n)}{n^{k}}+\bigo{x^{k-1}\sum_{n\leq \sqrt{x}}\dfrac{f(n)}{n^{k-1}}}-\epsilon(x)
				\end{align*}
			Now we need to focus on the following two sums.
				\begin{align*}
					\mathfrak{U}_{s}(x)
					& = \sum_{n\leq x}\dfrac{n^{s}}{n^{\xi}}\\
					\mathfrak{V}_{s}(x)
					& = \sum_{n\leq x}\dfrac{f(n)}{n^{s}}
				\end{align*}
			where $s\geq1$. Then
				\begin{align}
					(f\bullet g)(x)
					& = x^{\xi}\bigo{\mathfrak{U}_{k}(\sqrt{x})}+x^{k}\mathfrak{V}_{k}(\sqrt{x})+\bigo{x^{k-1}\mathfrak{V}_{k-1}(\sqrt{x})}-\mathfrak{B}(x)\label{eqn:diamond}
				\end{align}
			If $s+1<\xi$,
				\begin{align}
					\mathfrak{U}_{s}(x)
					& = \sum_{n\leq x}\dfrac{1}{n^{\xi-s}}\nonumber\\
					& = \dfrac{x^{1+s-\xi}}{1+s-\xi}+\zeta(\xi-s)+\bigo{x^{s-\xi}}\label{eqn:1}
				\end{align}
			If $s+1=\xi$,
				\begin{align*}
					\mathfrak{U}_{s}(x)
					& = \log{x}+C+\bigo{\dfrac{1}{x}}
				\end{align*}
			The case $\xi-1<s<\xi$ is possible if and only if $\xi$ is not an integer and $s=\floor{\xi}$ which can be taken care of in the same manner as \eqref{eqn:1}.  We can now assume $s\geq\xi$. In this case, $s-\xi\geq0$ and
				\begin{align*}
					\mathfrak{U}_{s}(x)
					& = \sum_{n\leq x}n^{s-\xi}\\
					& = \dfrac{x^{s-\xi+1}}{s-\xi+1}+\bigo{x^{s-\xi}}\label{eqn:6}\tag{6}
				\end{align*}
			For handling $\mathfrak{V}$, we will use Abel's partial summation formula.
			\begin{align*}
				\mathfrak{V}_{s}(x)
				& = \dfrac{F(x)}{x^{s}}+s\int_{1}^{x}F(t)t^{-s-1}dt\\
				& = \bigo{x^{\xi-s}}+s\bigo{\int_{1}^{x}t^{\xi-s-1}dt}
			\end{align*}
			Thus, we have
			\begin{align*}
				\mathfrak{U}_{s}(x),\mathfrak{V}_{s}(x)
				& = 
				\begin{cases}
					\log{x}+C+\bigo{\dfrac{1}{x}},\bigo{x}& \mbox{ if }s+1=\xi\\
					\dfrac{x^{1+s-\xi}}{1+s-\xi}+\zeta(\xi-s)+\bigo{x^{s-\xi}},\bigo{x^{\xi-s}}& \mbox{ if }s<\xi\mbox{ and }s+1\neq\xi\\
					x+\bigo{1},\bigo{\log{x}}& \mbox{ if }s=\xi\\
					\dfrac{x^{s-\xi+1}}{s-\xi+1}+\bigo{x^{s-\xi}}, \bigo{x^{\xi-s}}& \mbox{ if }s>\xi
				\end{cases}
			\end{align*}
			Plugging these back in \eqref{eqn:diamond}, we get the result.
		\end{proof}
	
		\begin{remark}
			The result in \autoref{thm:bullet} is a lot sharper than the trivial bound as noted below.
				\begin{align*}
					\sum_{n\leq x}f(n)\floor{\frac{x}{n}}^{k}
						& = \sum_{n\leq x}\func{f(n)}{\func{}{\frac{x}{n}}^{k}+\bigo{\func{}{\frac{x}{n}}^{k}}}\\
						& = x^{k}\sum_{n\leq x}\frac{f(n)}{n^{k}}+\bigo{x^{k-1}\sum_{n\leq x}\frac{f(n)}{n^{k-1}}}\\
				\end{align*}
			By Abel partial summation formula,
				\begin{align*}
					\sum_{n\leq x}\frac{f(n)}{n^{k}}
						& = \bigo{x^{\xi}}x^{-k}+k\int_{1}^{x}t^{\xi}t^{-k-1}dt\\
						& = \bigo{x^{\xi-k}}
				\end{align*}
			Thus we get the weaker bound $(f\bullet g)(x) = \bigo{x^{\xi}}$.
		\end{remark}
	\section{Asymptotic of A Generalization}
	Consider a generalization of the sum in \eqref{eqn:taugcdsum}.
		\begin{align*}
			\mathfrak{A}_{k}(x)
				& = \sum_{x_{1}\cdots x_{k}\leq x}F(\gcd(x_{1},\ldots,x_{k}))\\
		\end{align*}
	Fix a tuple $(x_{1},\ldots,x_{k})$ such that $\gcd(x_{1},\ldots,x_{k})=g$. Then running $g$ from $1$ to $x$,
		\begin{align*}
			\mathfrak{A}_{k}(x)
				& = \sum_{g\leq x}\sum_{\substack{x_{1}\cdots x_{k}\leq x\\\gcd(x_{1},\ldots,x_{k})=g}}F(g)
		\end{align*}
	Letting $x_{i}=y_{i}g$ with $\gcd(y_{1},\ldots,y_{k})=1$. 
		\begin{align*}
			\mathfrak{A}_{k}(x)
				& = \sum_{g\leq x}\sum_{g^{k}y_{1}\cdots y_{k}\leq x}F(g)\\
				& = \sum_{g^{k}\leq x}F(g)\sum_{\substack{y_{1}\cdots y_{k}\leq x/g^{k}\\\gcd(y_{1},\ldots,y_{k})=1}}1\\
				& = \sum_{g^{k}\leq x}F(g)\func{\mathfrak{B}_{k}}{\frac{x}{g^{k}}}
		\end{align*}
	Therefore, $\mathfrak{A}=F\circ\mathfrak{B}$ where
		\begin{align*}
			\mathfrak{B}_{k}(x)
				& = \sum_{\substack{x_{1}\cdots x_{k}\leq x\\\gcd(x_{1},\ldots,x_{k})=1}}1
		\end{align*}
	Letting $\mathfrak{T}_{k}(x)=\sum_{n\leq x}\tau_{k}(n)$ where
		\begin{align*}
			\tau_{k}(n)
				& = \sum_{x_{1}\cdots x_{k}=n}1
		\end{align*}
	we get the following using the principle of inclusion and exclusion
		\begin{align*}
			\mathfrak{B}_{k}(x)
				& = \sum_{n^{k}\leq x}\mu(n)\func{\mathfrak{T}}{\floor{\frac{x}{n^{k}}}}
		\end{align*}
	Then, we have the following by \eqref{eqn:geninv}.
		\begin{theorem}\label{thm:generalizationA}
			Let $F$ be an arithmetic function and $f$ be the M\"{o}bius inverse of $F$. Then $\mathfrak{A}=f\circ \mathfrak{T}$.
		\end{theorem}
	It is well known that (e.g. see \textcite{landau_1912_0})
		\begin{align}
			\mathfrak{T}_{k}(x)
				& = \frac{1}{(k-1)!}x\log^{k-1}{x}+O(x\log^{k-2}{x})\label{eqn:Tk}
		\end{align}
	Then
		\begin{align*}
			\mathfrak{A}_{k}(x)
				& = \sum_{n^{k}\leq x}f(n)\func{\mathfrak{T}_{k}}{\frac{x}{n^{k}}}\\
				& = \sum_{n^{k}\leq x}\func{f(n)}{\frac{1}{(k-1)!}\frac{x}{n^{k}}\log^{k-1}\frac{x}{n^{k}}+\bigo{x\sum_{n^{k}\leq x}\log^{k-2}\frac{x}{n^{k}}}}\\
				& = \dfrac{x\log^{k-1}x}{(k-1)!}\sum_{n\leq\sqrt[k]{x}}\dfrac{f(n)}{n^{k}}-\frac{xk}{(k-1)!}\sum_{n\leq\sqrt[k]{x}}\frac{f(n)\log^{k-1}n}{n^{k}}+\bigo{x\log^{k-2}x\sum_{n\leq\sqrt[k]{x}}\frac{f(n)}{n^{k}}}
		\end{align*}
	Taking $x\to\infty$, we finally have the asymptotic of $\mathfrak{A}$.
		\begin{theorem}\label{thm:asympA}
			Let $k\geq2$ be a positive integer. Then
				\begin{align*}
					\mathfrak{A}_{k}(x)
					& = \frac{x\log^{k-1}x}{(k-1)!}D_{k}(f)-\frac{xk}{(k-1)!}D_{k}^{k-1}(f)+\bigo{xD_{k}(f)\log^{k-2}x}
				\end{align*}
			where $D_{k}^{n}(f)$ is the $n$-th derivative of the Dirichlet series $D_{k}(f)$.
		\end{theorem}
	For the special case when $f(n)=\tau(n)$, we have $\mu\ast\tau=1$. Then we have the following.
		\begin{theorem}\label{thm:asympT}
			Let $k\geq2$ be a positive integer. Then
				\begin{align*}
					\sum_{x_{1}\cdots x_{k}\leq x}\tau(\gcd(x_{1},\ldots,x_{k}))
						& = \frac{x\log^{k-1}x}{(k-1)!}\zeta(k)-\frac{xk}{(k-1)!}\zeta^{k-1}(k)+O(x\zeta^{k-1}(k)\log^{k-2}x)
				\end{align*}
		\end{theorem}
	We can also say, the average order of $\tau(\gcd(x_{1},\ldots,x_{k}))$ is
		\begin{align*}
			\frac{\log^{k-1}{x}}{(k-1)!}+\bigo{\log^{k-2}{x}}
		\end{align*}
	Specifically, for the sum Dudek considered, $k=2,\zeta(2)=\frac{\pi^{2}}{6}$ and
		\begin{align*}
			\zeta^{1}(s)
				& = \zeta(s)\sum_{n\ge1}\frac{\Lambda(n)}{n^{s}}
		\end{align*}
	where $\Lambda(n)$ is the Von Mangoldt function. By Abel partial summation formula,
		\begin{align*}
			\sum_{n\leq x}\frac{\Lambda(n)}{n^{s}}
				& = \psi(x)\frac{1}{x^{s}}+s\int_{1}^{x}\psi(t)t^{-s-1}dt
		\end{align*}
	where $\psi(x)=\sum_{n\leq x}\Lambda(n)$ is Tschebischeff's second function. \textcite{tschebischeff_1852,tchebycheff_1852} proved that $\psi(x)=O(x)$ using elementary methods only. Using this, for $s>1$,
		\begin{align*}
			\zeta^{1}(s)
				& = \zeta(s)\lim_{x\to\infty}\sum_{n\leq x}\frac{\Lambda(n)}{n^{s}}\\
				& = \zeta(s)\bigo{x^{1-s}}
		\end{align*}
	Then, $\zeta^{1}(2)=\frac{\pi^{2}}{6}\bigo{\frac{1}{x}}$ and
		\begin{align*}
			\sum_{ab\leq x}
				& = \zeta(2)x\log{x}-2x\zeta^{1}(2)+O(x\zeta^{1}(2))\\
				& = \frac{\pi^{2}}{6}x\log{x}+\bigo{x\frac{\pi^{2}}{2}\bigo{\frac{1}{x}}}
		\end{align*}
	Thus, we have the following.
		\begin{theorem}\label{thm:asympD}
			$\sum_{ab\leq x} = \frac{\pi^{2}}{6}x\log{x}+\xi$ for a constant $\xi$.
		\end{theorem}
	\section{Asymptotic of Another Generalization}
	In our consideration of $\mathfrak{B}$, we considered the generalization of
		\begin{align*}
			\sum_{ab\leq x}\tau(\gcd(a,b))
				& = \sum_{n\leq\sqrt{x}}\tau(n)\sum_{\substack{uv\leq x\\\gcd(u,v)=1}}1
		\end{align*}
	Inspired this, in this section we consider the following generalization.
		\begin{align*}
			\mathfrak{S}_{k}(x)
				& = \sum_{\substack{x_{1}\cdots x_{k}\leq x\\\gcd(x_{i},x_{j})=1;i\neq j}}1
		\end{align*}
	Note that we can write it as
		\begin{align*}
			\mathfrak{S}_{k}(x)
				& = \mathfrak{S}_{k}(x-1)+\sum_{\substack{x_{1}\cdots x_{k}=\floor{x}\\\gcd(x_{i},x_{j})=1;i\neq j}}1\\
				& = \mathfrak{S}_{k}(x-1)+f_{k}(\floor{x})
		\end{align*}
	where $f_{k}(n)$ is the number of ways to write $n$ as a product of positive integers that are pairwise relatively prime. It is well known that $f_{k}(n)=k^{\omega(n)}$ where $\omega(n)$ is the number of distinct prime divisors of $n$. Then
		\begin{align*}
			\mathfrak{S}_{k}(x)
				& = \sum_{n\leq x}k^{\omega(n)}
		\end{align*}
	Fortunately, the asymptotic of $\mathfrak{S}$ can be easily established as a special case of the following elementary result \textcite[Thm. 1]{luca_toth_2017}.
		\begin{lemma}
			Let $k$ be a positive integer and $f(n)$ be a multiplicative function such that $f(p)=k$ and $f(p^{a})=a^{O(1)}$ for all primes $p$ and positive integers $a\geq2$. Then
				\begin{align*}
					\sum_{n\leq x}f(n)
						& = xC_{f}\log^{k-1}{x}+\bigo{x\log^{k-2}{x}}
				\end{align*}
			where
				\begin{align*}
					C_{f}
						& = \dfrac{1}{(k-1)!}\prod_{p\geq2}\func{}{1-\frac{1}{p}}^{k}\sum_{n\geq0}\frac{f(p^{n})}{p^{n}}
				\end{align*}
		\end{lemma}
	This result looks very similar to \autoref{thm:asympA}. Since $\omega(n)$ is additive, $f(n):= k^{\omega(n)}$ is multiplicative, $f(p)=k$, $f(p^{a})=k$,
		\begin{align*}
			C_{f}
				& = \frac{1}{(k-1)!}\prod_{p\geq2}\func{k}{1-\frac{1}{p}}^{k-1}
		\end{align*}
	so
		\begin{align*}
			\mathfrak{S}_{k}(x)
				& = C_{f}x\log^{k-1}{x}+\bigo{x\log^{k-2}{x}}
		\end{align*}
	\printbibliography
\end{document}